\documentclass[reqno]{amsart}

\usepackage[latin1]{inputenc}
\usepackage{amssymb}
\usepackage{graphicx}
\usepackage{amscd}
\usepackage[hidelinks]{hyperref}
\usepackage{color}
\usepackage{float}
\usepackage{graphics,amsmath,amssymb}
\usepackage{amsthm}
\usepackage{amsfonts}
\usepackage{latexsym}
\usepackage{epsf}
\usepackage{enumerate}
\usepackage{xifthen}
\usepackage{mathrsfs}
\usepackage{dsfont}
\usepackage{makecell}
\usepackage[FIGTOPCAP]{subfigure}
\usepackage{amsmath}
\allowdisplaybreaks[4]
\usepackage{listings}
\usepackage{etoolbox}
\usepackage{fancyhdr}

 \newtheoremstyle{mytheorem}
 {3pt}
 {3pt}
 {\slshape}
 {}
 {\bfseries}
 {.}
 { }
 {}

\numberwithin{equation}{section}

\theoremstyle{theorem}
\newtheorem{theorem}{Theorem}[section]

\newtheorem{lemma}[theorem]{Lemma}

\theoremstyle{definition}

\newtheorem{example}{Example}[section]

\newcommand{\doi}[1]{\href{http://dx.doi.org/#1}{doi: #1}}
\renewcommand{\MR}[1]{\href{http://www.ams.org/mathscinet-getitem?mr=#1}{MR#1}}

\newcommand{\Keywords}[1]{\ifthenelse{\isempty{#1}}{}{\smallskip \smallskip \noindent \textbf{Keywords}. #1}}
\newcommand{\MSC}[2][2010]{\ifthenelse{\isempty{#2}}{}{\smallskip \smallskip \noindent \textbf{#1MSC}. #2}}
\newcommand{\abstractnote}[1]{\ifthenelse{\isempty{#1}}{}{\smallskip \smallskip \noindent \textsuperscript{\dag}#1}}

\makeatletter
\def\specialsection{\@startsection{section}{1}%
  \z@{\linespacing\@plus\linespacing}{.5\linespacing}%
  {\normalfont}}
\def\section{\@startsection{section}{1}%
  \z@{.7\linespacing\@plus\linespacing}{.5\linespacing}%
  {\normalfont\scshape}}
\patchcmd{\@settitle}{\uppercasenonmath\@title}{\Large\boldmath}{}{}
\patchcmd{\@settitle}{\begin{center}}{\begin{flushleft}}{}{}
\patchcmd{\@settitle}{\end{center}}{\end{flushleft}}{}{}
\patchcmd{\@setauthors}{\MakeUppercase}{\normalsize}{}{}
\patchcmd{\@setauthors}{\centering}{\raggedright}{}{}
\patchcmd{\section}{\scshape}{\large\bfseries\boldmath}{}{}
\patchcmd{\subsection}{\bfseries}{\bfseries\boldmath}{}{}
\renewcommand{\@secnumfont}{\bfseries}
\patchcmd{\@startsection}{\@afterindenttrue}{\@afterindentfalse}{}{}
\patchcmd{\abstract}{\leftmargin3pc}{\leftmargin1pc}{}{}

\def\maketitle{\par
  \@topnum\z@ 
  \@setcopyright
  \thispagestyle{empty}
  \ifx\@empty\shortauthors \let\shortauthors\shorttitle
  \else \andify\shortauthors
  \fi
  \@maketitle@hook
  \begingroup
  \@maketitle
  \toks@\@xp{\shortauthors}\@temptokena\@xp{\shorttitle}%
  \toks4{\def\\{ \ignorespaces}}
  \edef\@tempa{%
    \@nx\markboth{\the\toks4
      \@nx\MakeUppercase{\the\toks@}}{\the\@temptokena}}%
  \@tempa
  \endgroup
  \c@footnote\z@
  \@cleartopmattertags
}
\makeatother

\newcommand{\la}{\lambda}
\newcommand{\og}{\overline{\mathcal{G}}}
\newcommand{\op}{\overline{\mathcal{P}}}

\newcommand{\ob}{\overline{\mathcal{B}}}
\newcommand{\ov}[1]{\overline{#1}}

\title[Overpartitions with bounded part differences]{Overpartitions with bounded part differences}

\author[S. Chern]{Shane Chern}
\address[S. Chern]{Department of Mathematics, The Pennsylvania State University, University Park, PA 16802, USA}
\email{shanechern@psu.edu; chenxiaohang92@gmail.com}

\author[A. J. Yee]{Ae Ja Yee}
\address[A. J. Yee]{Department of Mathematics, The Pennsylvania State University, University Park, PA 16802, USA}
\email{auy2@psu.edu}

\date{}

\begin{document}

{\footnotesize\noindent \textit{European J. Combin.} \textbf{70} (2018), 317--324. \MR{3779621}.\\
\doi{10.1016/j.ejc.2018.01.003}}

\bigskip \bigskip

\maketitle

\begin{abstract}

We generalize recent results of Breuer and Kronholm, and Chern on partitions and overpartitions with bounded differences between largest and smallest parts. We prove our generalization both analytically and combinatorially.  


\MSC{Primary 05A17; Secondary 05A19, 11P84.}

\end{abstract}


\section{Introduction}


A \textit{partition} of a positive integer $n$ is a non-increasing sequence of positive integers whose sum equals
 $n$. For example, there are three partitions of $3$: $3$, $2+1$, and $1+1+1$. 
 
In \cite{ABR2015}, Andrews, Beck and Robbins explored partitions with the difference between largest and smallest parts equal to $t$ for some positive integer $t$. 
 Motivated by their work, Breuer and Kronholm \cite{BK2016} studied the number of partitions of $n$ with the difference between  largest and smallest parts bounded by $t$, and they showed that the generating function  is
\begin{equation}\label{eq:g1}
\sum_{n\ge 1}p_t(n) q^n=\frac{1}{1-q^t}\left(\frac{1}{(q)_t}-1\right),
\end{equation}
where $p_t(n)$ counts the number of partitions of $n$ with the difference between largest and smallest parts bounded by $t$, and 
$$(a)_n=(a;q)_n:=\prod_{k=0}^{n-1} (1-a q^k).$$

The proof of Breuer and Kronholm has a geometric flavour, and their main tool used in the proof is polyhedral cones. 
Subsequently, Chapman \cite{Cha2016} also provided a simpler proof, which involves $q$-series manipulations.

An \textit{overpartition} of $n$ is a partition of $n$ in which the first occurrence of each distinct part may be overlined. For example, there are eight overpartitions of $3$: $3$, $\overline{3}$, $2+1$, $\overline{2}+1$, $2+\overline{1}$, $\overline{2}+\overline{1}$, $1+1+1$, and $\overline{1}+1+1$.  

Recently, motivated by the works of Andrews, Beck and Robbins,  Breuer and Kronholm, and Chapman,  the first author \cite{Che2017} considered an overpartition analogue  with bounded difference between largest and smallest parts.

To obtain a generating function analogous to \eqref{eq:g1}, apart from requiring the difference between  largest and smallest parts less than or equal to $t$, the first author added the following restriction: if the difference between largest and smallest parts is exactly $t$, then the largest part cannot be overlined. Let $g_t(n)$ count the number of such overpartitions of $n$. Then it was shown that
\begin{equation}\label{eq:g2}
\sum_{n\ge 1}g_t(n) q^n=\frac{1}{1-q^t}\left(\frac{(-q)_t}{(q)_t}-1\right).
\end{equation}
The first proof in his paper uses heavy $q$-series manipulation, which originates from \cite{ABR2015}. His second proof, which consists of many combinatorial ingredients such as the overpartition analogue of $q$-binomial coefficients introduced by Dousse and Kim \cite{DK2017}, however, still needs some nontrivial computation, and hence it is not completely combinatorial.


The main purpose of this paper is to provide a completely combinatorial and transparent proof of 
\eqref{eq:g2}.  More precisely, we prove the following refined result.

\begin{theorem}\label{th:1.1}
For a positive integer $t$, let $g_t(m,n)$ count the number of overpartitions of $n$ in which there are exactly $m$ overlined parts,  the difference between largest and smallest parts  is at most $t$, and if the difference between largest and smallest parts is exactly $t$, then the largest part cannot be overlined. Then
\begin{equation}\label{eq:th1.1}
\sum_{n\ge 1}\sum_{m\ge 0} g_t(m,n)z^mq^n=\frac{1}{1-q^t}\left(\frac{(-zq)_t}{(q)_t}-1\right).
\end{equation}
\end{theorem}

We remark that \eqref{eq:g1} and \eqref{eq:g2} follow immediately from \eqref{eq:th1.1} by taking $z\to 0$ and $z\to 1$, respectively. 

\subsection{Notation and terminology}

Throughout this paper, $\mathbb{Z}_{\ge 0}$ and $\mathbb{Z}_{>0}$ denote the set of nonnegative integers and positive integers, respectively. Given a partition or an overpartition $\la$ of $n$, let $\ell(\la)$ be the number of parts of $\la$ and $|\la|=n$ be the sum of the parts of $\la$. When $\la$ is an overpartition, we use $o(\la)$ to count the number of overlined parts in $\la$. We write parts in weakly decreasing order.

For a positive integer $t$, we denote by $\op_t$ the set of (nonempty) overpartitions with parts less than or equal to $t$ and no parts equal to $t$ overlined, and by $\og_t$ the set of (nonempty) overpartitions with the difference between largest and smallest parts at most $t$ and the largest part not overlined when the difference between largest and smallest parts is exactly $t$.  
Also,  $\ob_t$ denotes the set of bipartitions where the first subpartition, which can be an empty partition, consists of only parts equal to $t$, none overlined, and the second subpartition is a nonempty overpartition with parts less than or equal to $t$.  

The rest of this paper is organized as follows. In Section~\ref{sec2.1}, we first construct a weight preserving map $\phi$ from $\og_t$ to $\op_t$.  In Section~\ref{sec2.2}, we then construct another weight preserving map $\psi$ from $\op_t$ to $\ob_t$. Finally, by combining these two maps, we will deduce that $\og_t$ and $\ob_t$ have the same generating functions:
\begin{equation*}
\sum_{\pi\in \og_t} z^{o(\pi)} q^{|\pi |} = \sum_{\beta \in \ob_t} z^{o(\beta )} q^{|\beta |},
\end{equation*}
which is indeed equivalent to Theorem~\ref{th:1.1}. In Section~\ref{sec3}, a $q$-series proof of Theorem~\ref{th:1.1} will be given.

\section{A combinatorial approach}

\subsection{Partition sets $\og_t$ and $\op_t$} \label{sec2.1}

For an overpartition $\pi=(\pi_1,\pi_2,\ldots,\pi_\ell)$ in $\og_t$, let $s(\pi)=\lfloor \pi_{\ell}/t\rfloor $, where $\lfloor a\rfloor$ denotes the largest integer not exceeding $a$, and let $k(\pi)$ be the positive integer $k$ such that  $\pi_k\ge (s(\pi)+1)t$ and $\pi_{k+1}<(s(\pi)+1)t$. If there is no such $k$, then we let $k(\pi)=0$.  

We now define a map $\phi: \og_t\to \op_t$ as follows. For an overpartition $\pi \in \og_t$, 
let $\ell(\pi)=\ell$, $s(\pi)=s$ and $k(\pi)=k$.  Then
\begin{align*}
&\phi\colon (\pi_1,\pi_2,\ldots,\pi_\ell)\\
&\quad \mapsto (\underbrace{t,t,t,\ldots,t}_{\tiny \substack{s(\ell-k)+(s+1)k\\\text{times}}},\pi_{k+1}-st,\ldots,\pi_\ell-st,\pi_1-(s+1)t,\ldots,\pi_k-(s+1)t),
\end{align*}
where all the parts equal to $t$  are not overlined, and if $\pi_i$ is overlined, then $\pi_i-st$ (or $\pi_i-(s+1)t$ depending on the value of $i$) is overlined.
In other words, $\phi$ takes $\pi$ to $(t,t,\ldots,t,a_1,\ldots,a_\ell)$ where $a_1,\ldots,a_\ell$ are $\pi_1,\ldots,\pi_\ell$ reduced modulo $t$, cyclically permuted to make them weakly decreasing.

Here we note that there may be parts equal to $0$ in $\phi(\pi)$. If there are any parts equal to $0$, then we delete them so that $\phi(\pi)$ has positive parts only.

\begin{theorem}
$\phi$ is a weight preserving map from $\og_t$ to $\op_t$. 
\end{theorem}

\begin{proof}
Since $\pi_1-\pi_{\ell}\le t$, $s=\lfloor \pi_{\ell}/t \rfloor$, and $\pi_{k}\ge (s+1)t >\pi_{k+1}$, we have
\begin{equation*}
t >\pi_{k+1}-st \ge \cdots \ge \pi_{\ell}-st \ge  \pi_{1}-(s+1)t\ge  \cdots \ge \pi_k -(s+1)t.
\end{equation*}
Thus the parts of $\phi(\pi)$ are less than or equal to $t$, and if there are overlined parts, they are less than $t$. 

We now show that no more than one part of the same size is overlined. Since $\pi$ is an overpartition, at most one part of the same size is overlined in $\pi$. Hence, of $\pi_{1}-st,\ldots, \pi_{k}-st$, if there are overlined parts, then they must be of different sizes.  For the same reason, of $\pi_{k+1}-(s+1)t,\ldots, \pi_{\ell}-(s+1)t$, overlined parts  must be of different sizes.  Thus, if $\pi_{\ell}-st>\pi_1-(s+1)t$, then it is clear that all the overlined parts of $\phi(\pi)$ have different sizes. 

Let us suppose that $\pi_{\ell}-st=\pi_1-(s+1)t$. Then, we have $\pi_1-\pi_{\ell}=t$.  By the definition of $\og_t$, we know that all the parts equal to  $\pi_1$ are not overlined. Thus, for parts in $\phi(\pi)$ that are equal to $\pi_\ell-st=\pi_1-(s+1)t$, either the first occurrence or none may be overlined.
Therefore, $\phi(\pi)\in \op_t$.

We also note that the map $\phi$ preserves the weight of $\pi$, that is, $|\phi(\pi)|=|\pi|$. 
\end{proof}

As we see in the following example, the map $\phi$ is not a bijection. 
\begin{example}
Let $t=3$, $\pi=(7,\ov{4})$ and $\tilde{\pi}=(\ov{4},4,3)$. Then
\begin{gather*}
s(\pi)=1,\quad k(\pi)=1,\quad \phi(\pi)=(3,3,3,\ov{1},1), \quad |\phi(\pi)|=|\pi|=11;\\
s(\tilde{\pi})=1,\quad k(\tilde{\pi})=0,\quad \phi(\tilde{\pi})=(3,3,3,\ov{1},1), \quad |\phi(\tilde{\pi})|=|\tilde{\pi}|=11.
\end{gather*}
\end{example}
However, $\phi$ is a surjection since $\op_t$ is a subset of $\og_t$ and $\phi(\pi)=\pi$ for any $\pi\in \op_t$. So, we will count how many pre-images each $\mu\in \op_t$  has under $\phi$.

Let $\pi\in \og_t$. We describe how to recover $\pi$ from $\phi(\pi)$.   First, note that it is clear from the definition of $s(\pi)$ and $k(\pi)$ that   $\pi_i-(s(\pi)+1)t$ and $\pi_j-s(\pi) t$ are the remainders of $\pi_i$ and $\pi_j$ when divided by $t$ for  $1\le i \le k(\pi)$ and $ j> k(\pi)$. 
If the remainders are equal to $0$,  then they are deleted in $\phi(\pi)$.  Thus if we know the number of such deleted remainders, we can determine $\ell(\pi)$.  Also, one of the deleted remainders may have been overlined. 

 We then need to find $s(\pi)$ and  $k(\pi)$,  where $s(\pi)$ is the quotient of the smallest part of $\pi$ when divided by $t$ and $k(\pi)$ counts the number of parts whose quotients are equal to $s(\pi)+1$.  
Therefore, once we have $\ell(\pi)$, $k(\pi)$, and $s(\pi)$ along with the information on existence of an overlined deleted remainder, it is clear that we can recover $\pi$. Thus possible choices for $\ell(\pi)$, $k(\pi)$, and $s(\pi)$  with having a deleted remainder overlined or not will determine the number of pre-images under $\phi$. 
 
In the following lemma, we will see the range for $\ell(\pi)$.  For any $\mu\in\op_t$, we use $m(\mu)=m_t(\mu)$ to count the number of parts of $\mu$ equal to $t$.

\begin{lemma}\label{le:3.2}
Let $\pi$ be a nonempty overpartition in $\og_t$ and $\mu=\phi(\pi)$ in $\op_t$. Then we have
\begin{enumerate}[\indent\rm(i)]
\item $\ell(\pi)\le \ell(\mu)$;
\item $\ell(\pi)\ge \ell(\mu)-m(\mu) + \delta_{\ell(\mu), m(\mu)}$, where $\delta_{\ell(\mu), m(\mu)}$ is the Kronecker delta.  
\end{enumerate}
\end{lemma}

\begin{proof}
First, (i) is almost trivial. Under $\phi$, each part of $\pi$ splits into its residue modulo $t$ and as many $t$'s as the quotient, i.e.,  each part $\pi_i$ contributes $\lceil \pi_i/t\rceil$ to the number of parts of $\mu$. Thus $\ell(\pi)\le \ell(\mu)$. 

Next, we prove (ii). If all of the parts of $\mu$ are $t$, i.e.,  $\ell(\mu)=m(\mu)$,  then  
$$ \ell(\mu)- m(\mu)+ \delta_{\ell(\mu), m(\mu)}=1 \le \ell, $$
where the last inequality follows from the fact that $\pi$ is nonempty. 

We now suppose that $\mu$ has a part not equal to $t$, i.e, $\ell(\mu)-m(\mu)\ge 1$.  From the definition of $\phi$, we know that the parts of $\mu$ not equal to $t$ are the positive remainders of the parts of $\pi$, so at most $\ell$ parts of $\mu$ are not equal to $t$. Hence 
$$\ell(\mu)-m(\mu)+ \delta_{\ell(\mu), m(\mu)} = \ell(\mu)-m(\mu) \le \ell.$$ This completes the proof of (ii).
\end{proof}

It follows from Lemma~\ref{le:3.2} that 
\begin{equation}
\delta_{\ell(\mu), m(\mu)}  \le \ell(\pi)-\big(\ell(\mu)-m(\mu)\big) \le m(\mu), \label{remainder}
\end{equation}
where $\ell(\pi)-\big(\ell(\mu)-m(\mu)\big)$ is the number of  multiples of $t$ in $\pi$. 

\begin{lemma}\label{le:2.1}
Let $n$ be a fixed positive integer, and $n'$ a fixed nonnegative integer. Then the following system of equations
\begin{equation} 
\begin{cases}x+y &=n, \\s\, x+(s+1)y&=n'\end{cases} \label{eqs}
\end{equation}
has exactly one simultaneous solution $(x,y,s)\in \mathbb{Z}_{>0}\times\mathbb{Z}_{\ge0}\times\mathbb{Z}_{\ge0}$.
\end{lemma}

\begin{proof}
We readily see that $y=n'-s\, n$.  Also, since $x>0$ and $y\ge 0$, it follows from the first equation that  $0\le y<n$.  Hence 
$$\frac{n'}{n}-1<s\le \frac{n'}{n},$$
from which it follows that $s=\lfloor n'/n\rfloor$. Therefore, there is only one solution $(x, y, s)$.
\end{proof}

We are now ready to determine how many pre-images an overpartition in $\op_t$ has. 

\begin{theorem}\label{th:3.1}
Let $\mu$ be a nonempty overpartition in $\op_t$.  

{\rm (i)} If $\ell(\mu)=m(\mu)$, 
then  there are exactly $2m(\mu)$ pre-images in $\og_t$ under $\phi$.  Moreover, of those pre-images, exactly $m(\mu)$ pre-images have no overlined parts, and the other $m(\mu)$ pre-images have the first occurrence of the smallest parts overlined.

{\rm (ii)} If  $\ell(\mu)>m(\mu)$, 
then there are exactly $2m(\mu)+1$ pre-images in $\og_t$ under $\phi$. 
Moreover, of those pre-images, exactly $m(\mu)+1$ pre-images have the same number of overlined parts as $\mu$ and the other $m(\mu)$ pre-images have one more overlined part than $\mu$ does.
\end{theorem}

\begin{proof}
Let $\pi$ be a pre-image of $\mu$.  By Lemma~\ref{le:3.2}, we know that
\begin{equation}
 \ell(\mu)-m(\mu) +\delta_{\ell(\mu), m(\mu)} \le \ell(\pi)  \le \ell(\mu).\label{gap1}
\end{equation}
Hence, for any integer $\ell$ in this range,  we want to know how many $\pi\in\og_t$ with $\ell(\pi)=\ell$ can be pre-images of $\mu$. 

In order for $\pi$ to be a pre-image of $\mu$ with $\ell(\pi)=\ell$,  $s(\pi)$ and $k(\pi)$ must satisfy
\begin{equation}
s(\pi)(\ell-k(\pi))+(s(\pi)+1)k(\pi)=m(\mu). \label{eqs1}
\end{equation}
By the definition of $k(\pi)$, it should be less than $\ell(\pi)$, i.e.,  $\ell-k(\pi)>0$. Thus, \eqref{eqs1} is equivalent to that $\big(\ell-k(\pi), k(\pi), s(\pi) \big) $ is a solution to \eqref{eqs} with $n=\ell$ and $n'=m(\mu)$, which is unique. 

(i) Suppose that $\ell(\mu)=m(\mu)$. 
By \eqref{gap1},  there are $m(\mu)$ choices for $\ell$.  For a fixed $\ell$, $k(\pi)$ and $s(\pi)$ are uniquely determined as seen above.  With these $\big(\ell, k(\pi), s(\pi)\big)$, we can construct $\pi$, in which parts are multiples of $t$ differing by at most $t$  and there are no overlined parts. 

For each $\pi$, by having the first occurrence of the smallest parts overlined, we obtain a different pre-image. 
Therefore, the total number of pre-images must be equal to 
$2m(\mu)$ as claimed. Also, $m(\mu)$ pre-images have no overlined parts and the other $m(\mu)$ pre-images have one overlined smallest part. 

(ii) Suppose that 
$\ell(\mu)>m(\mu)$.   By \eqref{gap1},  there are $\big( m(\mu)+1\big)$ choices for $\ell$.  For a fixed $\ell$, $k(\pi)$ and $s(\pi)$ are uniquely determined.  With these $\big(\ell, k(\pi), s(\pi)\big)$, we can construct $\pi$, in which no multiples of $t$ are overlined. 

Note that from the construction of $\phi$,  $\ell(\mu)-m(\mu)$ counts the nonzero residues of the parts of $\pi$ modulo $t$. So,  if $\ell(\pi)>\ell(\mu)-m(\mu)$, then $\pi$ must have multiples of $t$ as parts.  For such $\pi$, by having the first occurrence of the smallest multiples of $t$ overlined, we obtain a different pre-image. 

Therefore, the total number of pre-images  must be equal to $\big(2m(\mu)+1\big)$ as claimed. Also, $\big(m(\mu)+1\big)$ pre-images have the same number of overlined parts as $\mu$ and the other $m(\mu)$ pre-images have one more overlined part than $\mu$ does.  
\end{proof}

Theorem \ref{th:3.1} yields
\begin{equation}
\sum_{\pi\in\og_t}z^{o(\pi)}q^{|\pi|}=\sum_{\mu \in\op_t}\left(\left(1-\delta_{\ell(\mu), m(\mu)}\right)+ (1+z) m(\mu)\right)z^{o(\mu)}q^{|\mu|}. \label{gen_qt_pt}
\end{equation}

In the following example, we present how to find all the pre-images $\pi$ of $\mu$. 
\begin{example} Let $t=3$. 

(i) Let $\mu=(3,3,3)$. Since $\ell(\mu)=m(\mu)=3$,  by Lemma~\ref{le:3.2} 
\begin{equation*}
1\le \ell(\pi) \le 3.
\end{equation*}
By solving \eqref{eqs1}, we have $\big(\ell(\pi), k(\pi), s(\pi)\big)=(1,0, 3), (2,1,1), (3, 0, 1)$, which yield
\begin{align*}
& (9), (\overline{9}), \\
&(6,3), (6, \overline{3}), \\
&(3, 3, 3), (\overline{3}, 3, 3), 
\end{align*}
respectively. There are $2m(\mu)$ pre-images.

(ii) Let $\mu=(3,3,3,\ov{1},1)$. Since $\ell(\mu)= 5$ and $m(\mu)=3$, 
by Lemma~\ref{le:3.2} 
\begin{equation*}
2\le \ell(\pi) \le 5.
\end{equation*}
By solving \eqref{eqs1}, we have $\big(\ell(\pi), k(\pi), s(\pi)\big)=(2,1,1),  (3,0, 1), (4, 3, 0), (5, 3 ,0)$, which yield
\begin{align*}
&(7, \overline{4}), \\
&(\overline{4}, 4, 3), (\overline{4}, 4, \overline{3}),\\
&(4, 3, 3, \overline{1}), (4, \overline{3}, 3, \overline{1}),\\
&(3, 3, 3, \overline{1}, 1),  (\overline{3}, 3, 3, \overline{1}, 1), 
\end{align*}
respectively.  Thus, there are $2m(\mu)+1$ pre-images. 
\end{example}

\subsection{Partition sets $\op_t$ and $\ob_t$}\label{sec2.2}

Let us recall the definition of $\ob_t$, from which it is clear that 
\begin{align}
\sum_{\beta\in\ob_t}z^{o(\beta)}q^{|\beta|}
&=(1+q^t+q^{2t}+\cdots)\left(\frac{(-zq)_t}{(q)_t}-1\right)\notag\\
&=\frac{1}{1-q^t}\left(\frac{(-zq)_t}{(q)_t}-1\right), \label{gen_bt}
\end{align}
where $o(\beta)$ denotes the number of overlined parts in $\beta$, which is indeed the number of overlined parts in the second subpartition of $\beta$. 

We now construct a map $\psi:\ob_t\to\op_t$ as follows:
\begin{enumerate}[\indent(1)]
\item First collect all parts equal to $t$ in both subpartitions and replace an overlined $t$ by a non-overlined $t$;
\item and then append the remaining parts in the second subpartition to the parts collected in (1).
\end{enumerate}
For example, $[(3),(3,3,\ov{1},1)]$ and $[(3),(\ov{3},3,\ov{1},1)]$ are both mapped to $(3,3,3,\ov{1},1)$ under $\psi$.

Let $\mu \in \op_t$.  Suppose that $\ell(\mu)=m(\mu)$, i.e., $\mu$ has parts equal to $t$ only. Then, its pre-image $\beta$ must be a bipartition of this form
\begin{equation*}
[(\underbrace{t,\ldots,t}_{m(\mu)-x}), (\underbrace{t,\ldots, t}_{x})]
\end{equation*}
for some $x>0$ with either the first occurrence or none of $t$'s in the second subpartition overlined.  Thus there are $2m(\mu)$ pre-images of $\mu$ in $\ob_t$ under $\psi$. Of those pre-images, $m(\mu)$ pre-images have the same number of overlined parts as $\mu$, and  the other $m(\mu)$ pre-images have one more overlined part than $\mu$.  

Suppose that $\ell(\mu)>m(\mu)$, i.e., $\mu$ has a part not equal to $t$. Then, its pre-image $\pi$ must be a bipartition of this form
\begin{equation*}
[(\underbrace{t,\ldots,t}_{m(\mu)-x}), (\underbrace{t,\ldots, t}_{x}, \mu_{m(\mu)+1},\ldots)]
\end{equation*}
for some $x\ge 0$ with either the first occurrence or none of $t$'s in the second subpartition overlined. Thus there are $2m(\mu)+1$ pre-images of $\mu$ in $\ob_t$ under $\psi$. Of those pre-images, $(m(\mu)+1)$ pre-images have the same number of overlined parts as $\mu$, and  the other $m(\mu)$ pre-images have one more overlined part than $\mu$. 


Therefore, it follows from the map $\psi$ that 
\begin{align}
\sum_{\mu \in\op_t}\left(\left(1-\delta_{\ell(\mu),m(\mu)}\right)+(1+z)m(\mu)\right)z^{o(\mu)}q^{|\mu|}&=\sum_{\beta\in\ob_t}z^{o(\beta)}q^{|\beta|}.
\label{gen_pt_bt}
\end{align}

By \eqref{gen_qt_pt},  \eqref{gen_bt}, and \eqref{gen_pt_bt}, 
\begin{equation*}
\sum_{n\ge 1} \sum_{m\ge 0} g_t(m,n)z^m q^n=\sum_{\pi\in \og_t} z^{o(\pi)} q^{|\pi |}=\sum_{\beta \in \ob_t} z^{o(\beta)} q^{|\beta |}=\frac{1}{1-q^t}\left(\frac{(-zq)_t}{(q)_t}-1\right),
\end{equation*}
which completes the proof of Theorem~\ref{th:1.1}.

\section{Final remarks} \label{sec3}

We remark that, by slightly modifying the first proof of \cite[Theorem 2.1]{Che2017}, we can also prove Theorem \ref{th:1.1} analytically.

Let
$${}_{r+1}\phi_s\left(\begin{matrix} a_0,a_1,a_2\ldots,a_r\\ b_1,b_2,\ldots,b_s \end{matrix}; q, z\right):=\sum_{n\ge 0}\frac{(a_0;q)_n(a_1;q)_n\cdots(a_r;q)_n}{(q;q)_n(b_1;q)_n\cdots (b_s;q)_n}\left((-1)^n q^{\binom{n}{2}}\right)^{s-r}z^n.$$
Then we will need the following identities later.

\begin{lemma}[First $q$-Chu--Vandermonde Sum {\cite[Eq.~(17.6.2)]{And2010}}]\label{le:chu}
We have
\begin{equation}\label{eq:chu}
{}_{2}\phi_{1}\left(\begin{matrix} a,q^{-n}\\ c \end{matrix}; q, cq^{n}/a\right)=\frac{(c/a;q)_n}{(c;q)_n}.
\end{equation}
\end{lemma}

\begin{lemma}[{\cite[Eq.~(17.9.6)]{And2010}}]\label{le:32}
We have
\begin{equation}\label{eq:32}
{}_{3}\phi_{2}\left(\begin{matrix} a,b,c\\ d,e \end{matrix}; q, de/(abc)\right)=\frac{(e/a;q)_\infty(de/(bc);q)_\infty}{(e;q)_\infty(de/(abc);q)_\infty} {}_{3}\phi_{2}\left(\begin{matrix} a, d/b, d/c\\ d, de/(bc) \end{matrix}; q, e/a\right).
\end{equation}
\end{lemma}

First, note that the generating function for partitions in $\og_t$ with smallest part equal to $r$ is
$$\frac{(1+z)q^r}{1-q^r}\frac{1+zq^{r+1}}{1-q^{r+1}}\cdots\frac{1+zq^{r+t-1}}{1-q^{r+t-1}}\frac{1}{1-q^{r+t}},$$
in which the coefficient of $z^m q^n$ counts the number of such overpartitions of $n$ with exactly $m$ overlined parts. Hence
\begin{align}
\sum_{n\ge 1}\sum_{m\ge 0}g_t(m,n)z^mq^n&=\sum_{r\ge 1}\frac{(1+z)q^r}{1-q^r}\frac{1+zq^{r+1}}{1-q^{r+1}}\cdots\frac{1+zq^{r+t-1}}{1-q^{r+t-1}}\frac{1}{1-q^{r+t}} \notag\\
&=(1+z)\sum_{r\ge 1}\frac{(q)_{r-1}(-zq)_{r+t-1}}{(q)_{r+t}(-zq)_r}q^r \notag\\
& =(1+z)q\sum_{r\ge 0}\frac{(q)_{r}(-zq)_{r+t}}{(q)_{r+t+1}(-zq)_{r+1}}q^r \notag\\
&=\frac{(1+z)q(-zq)_{t}}{(1+zq)(q)_{t+1}}\sum_{r\ge 0}\frac{(q)_{r}(q)_{r}(-zq^{t+1})_{r}}{(q)_r(q^{t+2})_{r}(-zq^2)_{r}}q^r \notag \\
&=\frac{(1+z)q(-zq)_{t}}{(1+zq)(q)_{t+1}}\ {}_{3}\phi_{2}\left(\begin{matrix} q,q,-zq^{t+1}\\ -zq^2,q^{t+2} \end{matrix}; q, q\right) \notag \\
& = \frac{(1+z)q(-zq)_{t}}{(1+zq)(q)_{t+1}}\frac{(q^{t+1})_\infty(q^2)_\infty}{(q^{t+2})_\infty(q)_\infty}\ {}_{3}\phi_{2}\left(\begin{matrix} q,-zq,q^{1-t}\\ -zq^2,q^{2} \end{matrix}; q, q^{t+1}\right) \tag{by Eq. \eqref{eq:32}}\\
&=  \frac{(1+z)q(-zq)_{t}}{(1-q)(1+zq)(q)_{t}} \sum_{r\ge 0}\frac{(-zq)_{r}(q^{1-t})_{r}}{(-zq^2)_{r}(q^2)_r}q^{r(t+1)} \notag\\
&=  -\frac{(-zq)_{t}}{(1-q^t) (q)_{t}} \sum_{r\ge 0}\frac{(-z)_{r+1} (q^{-t})_{r+1}  }{(-zq)_{r+1}(q)_{r+1}}q^{(r+1)(t+1)} \notag\\
&=  -\frac{(-zq)_{t}}{(1-q^t)(q)_{t}} \left({}_{2}\phi_{1}\left(\begin{matrix} -z,q^{-t}\\ -zq \end{matrix}; q, q^{t+1}\right)-1\right) \notag \\
&=  -\frac{(-zq)_{t}}{(1-q^t)(q)_{t}} \left(\frac{(q)_t}{(-zq)_t}-1\right) \tag{by Eq. \eqref{eq:chu}}\\
&= \frac{1}{1-q^t}\left(\frac{(-zq)_t}{(q)_t}-1\right). \notag
\end{align}

\subsection*{Acknowledgements}

We thank the referees for their careful reading and helpful comments. The second author was partially supported by a grant ($\#$280903) from the Simons Foundation.

\bibliographystyle{amsplain}

\end{document}